\documentclass{amsart}
\usepackage[T1]{fontenc}   
\usepackage{graphicx}

\newtheorem{corollary}{Corollary}[section]
\newtheorem{theorem}{Theorem}[section]

\newtheorem{proposition}[theorem]{Proposition}
\theoremstyle{definition}
\newtheorem{definition}[theorem]{Definition}

\newtheorem{problem}[theorem]{Problem}
\newtheorem{problema}[theorem]{General Problem}
\newtheorem{remark}[theorem]{Remark}
\numberwithin{equation}{section}

\begin{document}

\vspace{0.5in}

\renewcommand{\bf}{\bfseries}
\renewcommand{\sc}{\scshape}
\vspace{0.5in}

\title[Extended $b$-metric-preserving functions]
{Extended $b$-metric-preserving functions\\}

\author{Reinaldo Mart\'{i}nez Cruz}
\address{Licenciatura en Matem\'aticas Aplicadas, UATX} 
\email{reinaldomar1964@hotmail.com}

\author{Marian Citlalli Cruz Cruz}
\address{Licenciatura en Matem\'aticas Aplicadas, UATX}
\email{maristar1943@gmail.com}


\thanks {ALL references are real and correct; ALL citations are imaginary.}

\begin{abstract}
In this paper, we introduce a couple of classes of functions, denoted by $\mathcal{DU}$ and $\mathcal{EB}$. We present the relationship between them and other known classes. Also, we show that the elements of the class $\mathcal{EB}$, are amenable and quasi-subadditive functions (Theorem \ref{thEB}). Finally, in the Theorem  \ref{thm20}, we establish that the graphic of these elements is contained in the region proposed by J. Dobos and Z. Piotrowski(see \cite{D.P}).

\end{abstract}

\maketitle

\section{\bf Introduction}

Let $f:[0,\infty)\to [0,\infty)$. We said $f$ metric-preserving, if for each  metric space $(X,d)$,  $f\circ d$ is a metric.
This notion appears for the first time in the article \cite{W.A}  and from that moment on it is investigated by several authors; see for example, Bors\'ik and Dobous \cite{B.D},  Paul Corazza \cite{C.P}, Pongsriiam and Termwuttipong \cite{P.I},  Kamran and Samreen and Q. UL Ain \cite{T.Kam}, Khemaratchatakumthorn and  Pongsriiam \cite{T.K} and \cite{T.K.P}.

In this contex, the next question is natural.

\begin{problem}
Under what conditions on a function  $f:[0,\infty)\to [0,\infty)$
is it the case that for every metric space $(X,d)$, $f\circ d $ is still a metric?
\end{problem}

Curently the metric notion has several generalizations (and thus metric space); many of them has be obtained after to give a slight genaralization to triangle inequality ((M3) in Definition 2.1). For instance Czerwik in \cite{C.S} gives a weaker axiom than the triangular inequality and formally defines a $b$--metric space in order to generalize the Banach contraction mapping theorem.
Later Fagin in \cite{Fagin} discussed some kind of relaxation in triangular inequality and called this new distance measure as non-linear elastic mathing (NEM).

Similar type of relaxed triangle inequality was also used for trade
measure \cite{Cortelazzo} and to measure ice floes \cite{McConnell}.
All these applications intrigued and pushed us to introduce the concept of extended $b$--metric space.
So that the results obtained for such rich spaces become more viable in different directions of applications.

From what has been previously commented, we have the following general quesion.

\begin{problema}
Let $f:[0,\infty)\to [0,\infty)$ a function and

$\mathcal{P}\in \{\text{metric}, \text{ultrametric}, \text{weak ultrametric}, \text{extended $b$--metric}\}$.
Suppose that $(X,d)$ is such that $d\in \mathcal{P}$.
Under what conditions on the function  $f$
is it the case that $f\circ d\in \mathcal{P}$?
\end{problema}

J. Dobos and Z. Piotrowski in \cite{D.P} present a region in the plane, where the graph of any distance-preserving,
inter-valued function $f$ is contained.

In this paper, we introduce a couple of classes of functions
namely, the class of functions that:
(1) preserving the weak ultrametric, denoted by $\mathcal{DU}$ and  (2) preserve the extended $b$--metric, denoted by $\mathcal {EB}$.
In Proposition \ref{pro1}, we show that the collection $\mathcal{U}$ is contained in $\mathcal{DU}$.
Also, we will see that the class $\mathcal{DU}$ is contained in the class $\mathcal{EB}$ (see Theorem \ref{th2}).
From this fact, it follows that the family $\mathcal{EB}$ contains all the families given in \cite{T.K} and \cite {T.K.P}.
We show that Theorems \ref{th2}, \ref{thEB} and \ref{thG1} are generalizations of the results given in \cite{T.K} and \cite{T.K.P}.
Finally, with the Theorem \ref{thm20}, we verify that, the graph of any function with integer values that extended $b$--metric-preserving is in the region proposed by J. Dobos and Z. Piotrowski in \cite{D.P}.

\section{Preliminares}

With the purpose of making this work self-contained, we are going to briefly expose the results and definitions necessary for the reading of this work. The interested reader can consult the references \cite{T.Kam}, \cite{T.K} and \cite{T.K.P}.
\begin{definition}
Let $X$ be a nonempty set. A function $d: X \times X \to [0.\infty)$
is called a {\sl metric} if for all $x,y,z\in X$ it satisfies:

{\rm (M1)} $d(x,y)= 0$  if and only if $ x=y$,

{\rm (M2)} $d(x,y)=d(y,x)$,

{\rm (M3} $d(x,y)\leq d(x,z)+ d(z,y)$.

\end{definition}

It is well known that the notion of metric currently has various generalizations; among other:

\begin{definition}(\cite{P.Fraigniaud})\label{ultrame}
Let $X$ be a nonempty set and  $d: X \times X \to [0.\infty)$  a function. We say that

{\rm (A)} $d$ is a {\sl ultrametric} if for all $x,y,z\in X$ it satisfies:

{\rm (U1)} $d(x,y)= 0$  si  y  s\'olo  si $ x=y$,

{\rm (U2)} $d(x,y)=d(y,x)$,

{\rm (U3)} $d(x,y)\leq \max\{d(x,z), d(z,y)\}$.

{\rm (B)} $d$ is a {\sl weak ultrametric} if for all $x,y,z\in X$ it satisfies:
{\rm (U1)}, {\rm (U2)}  and

{\rm (I3)} there exists $C\geq 1$ such that  $d(x,y)\leq C\max\{d(x,z), d(z,y)\}$.

The pair $(X,d)$  is called a ultrametric space when $d$ is a ultrametric in $X$  (or $(X,d)$  is called a weak ultrametric space when $d$ is a weak ultrametric in $X$).
\end{definition}

Ultrametric spaces originate in the studio of $p$-adic numbers and
nonarchimedean analysis  \cite{S.B} y \cite{E.Y}, topology and dynamical
system \cite{M.C}, topological algebra \cite{J.P}, and theoretical
computer science \cite{S.P}.

\begin{definition}({\rm \cite{B.I.A}})\label{b-me}
Let $X$ be a nonempty set. A function $d: X \times X \to [0.\infty)$
is called a $b-$metric  if for all $x,y,z\in X$ it satisfies:

{\rm (B1)} $d(x,y)= 0$  if and only if $ x=y$,

{\rm (B2)} $d(x,y)=d(y,x)$,

{\rm (B3)} there exist $s\geq 1$ such that $d(x,y)\leq s(d(x,z)+ d(z,y))$ ($s-$triangle inequality).

If $d$ is a $b-$metric on $X$, then $(X,d)$ is called a $b-$metric space.
\end{definition}

\begin{remark}\label{ob0} It follows from the definition that:

{\rm (i)} every ultrametric space  $(X,d)$ is a weak ultrametic space,

{\rm (ii)} every metric space $(X,d)$ is a $b$--metric space.
\end{remark}

\begin{definition}({\rm \cite{T.Kam}})\label{b-me}
Let $X$ be a nonempty set and $\theta:X\times X\to [1,\infty)$.
A function $d_{\theta}: X \times X \to [0.\infty)$
is called an extended $b$--metric if for all $x,y,z\in X$ it satisfies:

{\rm ( $d_{\theta}1$ )} $ d_{\theta}(x,y)=0 $  if and only if $ x=y$,

{\rm ( $d_{\theta}2$ )} $ d_{\theta}(x,y)=d_{\theta}(y,x) $,

{\rm ( $d_{\theta}3$ )} $ d_{\theta}(x,z)\leq \theta(x,z) ( d_{\theta}(x,y)+ d_{\theta}(y,z) )$.
The pair $(X,d_{\theta})$ is called an extended $b$--metric space.
\end{definition}

\begin{remark}\label{ob2} If for any  $x,\ y\in X$,  $\theta(x,y)=s$, for some $s\geq 1$ then we obtain the definition of a $b$-metric space.
\end{remark}

The existence of several generalizations to the notion of metric along with the notion of function that preserves the metric, leads naturally to generalization of the concept of function that preserves the metric.

\begin{definition} Let $f:[0,\infty)\to [0,\infty)$ a function and

$\mathcal{P}\in \{\text{metric}, \text{ultrametric}, \text{weak ultrametric}\}$. We say that
 $f$ is $\mathcal{P}$-preserving,  if  for all metric space $(X,d)$, with $d\in\mathcal{P}$, $f\circ d\in\mathcal{P}$;

\end{definition}

In this paper we will focus our attention on the cases \text{ultrametric}, \text{weak ultrametrica} and \text{extended $b$--metric}. We will say that:

\medskip

$(A)$ $f$ {\sl ultrametric-preserving}, if for all ultrametric $(X,d)$ space, $f\circ d$ is an ultrametric, and we denote the set of all {\sl ultrametric-preserving-functions} to the class $\mathcal{U}$,

$(B)$ $f$ {\sl weak ultrametric-preserving}, if for all weak ultrametric $(X,d)$ space, $f\circ d$ is an weak ultrametric, and we denote the set of all {\sl weak ultrametric-preserving-functions} to the class $\mathcal{DU}$.

$(C)$ $f$ {\sl extended $b$--metric-preserving}, if for all extended $b$--metric $(X,d_{\theta})$ space, there exits an $\hat{\theta}:X\times X\to [1,\infty)$ such that $(f\circ d_{\theta})_{\hat{\theta}}$ is an extended $b$--metric, and we denote the set of all {\sl extended $b$--metric-preserving-functions} to the class $\mathcal{EB}$.

$(D)$ $f$ {\sl metric-preserving}, if for all metric $(X,d)$ space, $f\circ d$ is an metric, and we denote the set of all {\sl metric-preserving-functions} to the class $\mathcal{M}$,

$(E)$ $f$ {\sl $b$--metric-preserving}, if for all $b$-metric space $(X,d)$, $f\circ d$ is an $b$--metric on $X$, and we denote the set of all {\sl $b$--metric-preserving-functions} to the class $\mathcal{B}$,

$(F)$ $f$ {\sl metric-$b$--metric-preserving}, if for all metric spaces $(X,d)$, $f\circ d$ is a $b$-metric on $X$, and we denote the set of all {\sl metric-$b$--metric-preserving-functions} to the class $\mathcal{MB}$, and

$(G)$ $f$ {\sl $b$--metric-metric-preserving}, if for all $b$-metric spaces $(X,d)$, $f\circ d$ is a metric on $X$, and we denote the set of all {\sl $b$--metric-metric-preserving-functions} to the class $\mathcal{BM}$.
\medskip

The following results establishes the relationship between classes given in the previous definition.

\begin{theorem}\label{thkp}
The following relationship are satisfied.
\begin{enumerate}
\item (\cite{T.K}) $\mathcal{BM}\subseteq \mathcal{M}\subseteq \mathcal{B} \subseteq \mathcal{MB}$, $\mathcal{M}\not\subseteq \mathcal{BM}$ and $\mathcal{B}\not\subseteq \mathcal{M}$.
\item (\cite{T.K.P}) $\mathcal{MB}=\mathcal{B}$.
\item (\cite{T.K.P}) $\mathcal{B}=\mathcal{DU}$.
\end{enumerate}
\end{theorem}

Next, we will establish that all function which preserves the ultrametric, preserves the weak ultrametric too.

\begin{proposition}\label{pro1}
We have  $\mathcal{U}\subseteq \mathcal{DU}$.
\end{proposition}

\begin{proof}
Let $f\in \mathcal{U}$, then for each ultrametric space $(X,d)$, $f\circ d$ is a ultrametric.
By Remarks \ref{ob0}, $(X,d)$ is a weak ultrametric space.
Suppose $f\circ d$ is not a weak ultrametric in $X$, then for each $C\geq 1$ there exist points $x,y,z\in X$ such that
\[
(f\circ d)(x,y)>C\max\{(f\circ d)(x,z), (f\circ d)(z,y)\}.
\]
In particular, for $C=1$, there exist points $x,y,z\in X$ such that
\[
(f\circ d)(x,y)>\max\{(f\circ d)(x,z), (f\circ d)(z,y)\}.
\]
We obtained $(X,f\circ d)$ is not ultrametric.
\end{proof}

As a consequence of Proposition \ref{pro1} and Theorem \ref{thkp}, we obtain the following result.

\begin{corollary}\label{cor15}
$\mathcal{U}\subseteq \mathcal{DU}=\mathcal{B}=\mathcal{MB}$.
\end{corollary}

From now on, we will focus our attention on the class $\mathcal{EB}$.
In the next result we establish the relationship that there exist between this class and the others given above.

\begin{theorem}\label{th2}
The $\mathcal{EB}$ class contains to $\mathcal{B}$ class.
\end{theorem}

\begin{proof} Let $f\in \mathcal{B}$ and let $d_{\theta}$  be a extended $b-$metric on a space $X$.
We will show that, exist $\hat{\theta}:X\times X\to [1,\infty)$ such tha $(f\circ d_{\theta})_{\hat{ \theta}}$ is an extended $b-$metric on $X$.

Since $f\in \mathcal{B}$, then  $f$ is amenable y quasi-subadditive (see  Theorem \ref{thkp} and {\rm \cite[Theorem 20]{T.K}}.
Therefore for every $x,y\in X$,
\[
(f\circ d_{\theta})_{\hat{ \theta}}(x,y)=0\quad {\rm if}\quad {\rm and}\quad{\rm only}\quad{\rm if}\quad x=y.
\]
The condition $d_{\theta}2$ is obvious. So it remains to show that $(f\circ d_{\theta})_{\hat{ \theta}}$ satisfies $d_{\theta}3$.
Since $f$ is quasi-subaditive, there exists a constant $s\geq 1$ such that
\[
f(x+y)\leq s(f(x) + f(y)) \quad{\rm for}\quad{\rm all}\quad x,y\in [0,\infty).
\]

Let $\theta:X\times X\to [1,\infty)$  defined by
$\theta(x,y)= s$ for every  $x,y\in X$.
Let $d_{\theta}$ the usual metric on $\mathbb{R}$.
\[
\hspace{-37mm}(f\circ d_{\theta})_{\hat{\theta}}(0,x+y)=f(x+y)
\]
\[
\leq s(f(x)+f(y))
\]
\[
\hspace{8mm}\leq \theta (x,z)(f(x) + f(y))
\]
\[
\hspace{33mm}=\theta(x,y)(f(d_{\theta}(0,x)) + f(d_{\theta}(x+y,x)))
\]
\[
\hspace{43mm}=\theta(x,y)((f\circ d_{\theta})_{\hat{\theta}}(0,x) + (f\circ d_{\theta})_{\hat{\theta}}(x+y,x)).
\]
\end{proof}

\begin{corollary}
If $\mathcal{P}\in\{\mathcal{M}, \mathcal{U}, \mathcal{DU}, \mathcal{MB}\}$, then $\mathcal{P}\subseteq \mathcal{EB}$.
\end{corollary}

\begin{definition}

Let $f:[0,\infty)\to [0,\infty)$. We said

{\rm (a)} (\cite{D}) $f$ is amenable if and only if $f^{-1}(\{0\})=\{0\}$,

{\rm (b)} (\cite{Rosenbaum}) $f$ is subadditive if  for all $x,y\in [0,\infty)$,   $f(x+y)\leq f(x)+f(y)$,

{\rm (c) } (\cite{T.K})$f$ is $quasi$-subadditive if there exists $s\geq 1$
 such that $f(a + b)\leq s(f(a) + f(b))$ for all $a, b\in [0, \infty)$.
\end{definition}

\begin{definition}
{\rm a)}(\cite{Terpe})   A triangle triplet, is a triple $(a, b, c)$ con $a, b, c\geq 0$ such that $a\leq b+c$, $b\leq a+c$ y $c\leq a+b$,

{\rm b)}(\cite{T.K.P}) Let $s\geq 1$ and $a, b, c \geq 0$. A triple $(a, b, c)$ is said to be an $s$-triangle triplet if $a \leq s(b + c)$, $b\leq s(a + c)$, and $c \leq s(a + b)$.

{\rm c)} Let $\theta:X\times X\to [1,\infty)$ and $a, b, c \geq 0$. A triple $(a, b, c)$ is said to be an $\theta$-triangle triplet if $a \leq \theta(x,y)(b + c)$, $b\leq \theta(x,z)(a + c)$, and $c \leq \theta(z,y)(a + b)$, for all $x,y,z\in X$.

We denote for $\triangle$, $\triangle_{s}$ and $\triangle_{\theta}$ be the set of all triangle triplets, $s$-triangle triplets and $\theta$-triangle triplets respectively.
\end{definition}

The following statements are easy to verify.

\begin{remark}\label{obs5}

{\rm (i)} If $f:[0,\infty)\to [0,\infty)$ is subadditive,
 then  $f$ is $quasi$-subadditive.

{\rm (ii)} For all $(a, b, c)\in \triangle_{s}$, we obtain $(a, b, c)\in \triangle_{\theta}$.
\end{remark}

In the following result, we will show that the elements of the $\mathcal{EB}$ class satisfy to be amenable and quasi-subadditive.

\begin{theorem}\label{thEB}
If $f\in \mathcal{EB}$ then $f$ is a amenable and quasi-subadditive.
\end{theorem}
\begin{proof}

Assume $f\in \mathcal{EB}$ y let $d_{\theta}(x,y)=\vert y-x \vert$ for all $x,y\in \mathbb{R}$.
Then $(f\circ d_{\theta})_{\hat{ \theta}}$ is an extended $b-$metric on $\mathbb{R}$.
Then
\[
f(0)= f(d_{\theta}(0,0))=(f\circ d_{\theta})_{\hat{\theta}}(0,0)=0.
\]
Suppose $x\in [0,\infty)$ and $f(x)=0$.
Then
\[
0=f(x)=f(d_{\theta}(0,x))= (f\circ d_{\theta})_{\hat{ \theta}}(0,x).
\]
Since $(f\circ d_{\theta})_{\hat{ \theta}}(0,x)=0$ and $(f\circ d_{\theta})_{\hat{ \theta}}$ is a extended $b$--metric, we have $x=0$.
This shows that $f$ is amenable. Next, since $f$ extended $b$--metric-preserving,  there exists $\hat{\theta}:X\times X\to [1,\infty)$ such that, for all $x,y,z\in \mathbb{R}$,
\[
(f\circ d_{\theta})_{\hat{ \theta}}(x,y)\leq \theta(x,y)((f\circ d_{\theta})_{\hat{ \theta}}(x,z)+(f\circ d_{\theta})_{\hat{ \theta}}(z,y)).
\]
To show that $f$ is quasi-subadditive, lets $a,b\in [0,\infty)$ and $\theta(0,a+b)=s$. We have $s\geq 1$ and
\[
\hspace{-57mm}f(a+b)=f(d_{\theta}(0,a+b))
\]
\[
\hspace{-38mm}= (f\circ d_{\theta})_{\hat{ \theta}}(0,a+b)
\]
\[
\hspace{6mm}\leq \theta(0,a+b)((f\circ d_{\theta})_{\hat{ \theta}}(0,a)+(f\circ d_{\theta})_{\hat{ \theta}}(a,a+b))
\]
\[
\hspace{-42mm}=s(f(a)+f(b)).
\]
This is $f(a+b)\leq s(f(a)+f(b))$. This shows that $f$ is quasi-subadditive.
\end{proof}

\begin{corollary}\label{corEB}
If $f\in \mathcal{B}$ then $f$ is a  quasi-subadditive.
\end{corollary}
\begin{proof}
Assume $f\in \mathcal{B}$. Then by the Theorem \ref{th2}, $f\in \mathcal{EB}$, and by the Theorem \ref{thEB}, $f$ es amenable and  quasi-subadditive.
\end{proof}

Of course, a natural question is:

\begin{problem}
If $f:[0,\infty)\to [0,\infty)$ is amenable and subadditive, then  $f$ is a element of $\mathcal{EB}$?
\end{problem}

So far the author does not know an answer to the above question; however, if we add as a hypothesis that $f$ is increasing, the answer to the previous question is yes.

\begin{theorem}\label{thG1}
Let $f:[0,\infty)\to [0,\infty)$. If $f$ is amenable, quasi-subadditive, and increasing on $[0,\infty)$, then $f\in \mathcal{EB}$.
\end{theorem}

\begin{proof}
Assume that $f$ is amenable, quasi-subadditive, and increasing on $[0,\infty)$.
Let $(X,d_{\theta})$ be a extended $b$--metric space.
Show that, there exist $\hat{\theta}:X\times X\to [1,\infty)$ such that $(f\circ d_{\theta})_{\hat{ \theta}}$ is an extended $b$--metric on $X$.

As $f$ is amenable, $0=(f\circ d_{\theta})_{\hat{ \theta}}(x,y)$ if and only if $x=y$.
The property $d_{\theta}2$  is easy to verify.
Since $f$ is quasi-subadditive, there exist $s\geq 1$ such that
\begin{equation}\label{eq0}
f(a+b)\leq s(f(a)+f(b))\quad {\rm for}\quad{\rm all}\quad a,b\in [0,\infty).
\end{equation}

Lets $x,y,z\in X$, and  $\theta(x,y)=s$.  We have to $\theta(x,y)\geq 1$.
We will see that the property $d_{\theta}3$ is satisfied.

By inequality \eqref{eq0}, and the hypothesis that $f$ is increasing, we obtain
\[
\hspace{-40mm}(f\circ d_{\theta})_{\hat{\theta}}(x,y)=f(d_{\theta}(x,y))
\]
\[
\hspace{-2mm}\leq f(d_{\theta}(x,z)+d_{\theta}(z,y))
\]
\[
\hspace{15mm}\leq \theta(x,y)(f(d_{\theta}(x,z))+f(d_{\theta}(z,y)))
\]
\[
\hspace{26mm}=\theta(x,y)((f\circ d_{\theta})_{\hat{\theta}}(x,z)+(f\circ d_{\theta})_{\hat{\theta}}(z,y)).
\]
This implies that $f\in \mathcal{EB}$.
\end{proof}

Applying the triangular inequality and the previous definition we obtain.
\begin{remark}\label{obs4} If $(X, d_{\theta})$ is a extended $b$--metric spaces, then
\[
(d(x,y), d(y,z), d(x,z))\in \triangle_{\theta}\quad{\rm for}\quad{\rm all}\quad x,y,z\in X.
\]
\end{remark}

To prove Theorem \ref{th6}, the following proposition is useful.

\begin{proposition}\label{pro0} ([15]) Let $a$, $b$ and $c$ be positive real numbers.
Then $(a,b,c)\in\triangle$ iff there are $u,v,w\in \mathbb{R}^{2}$, $u\neq v\neq w$, such that
$a=d(u,v)$, $b=d(u,w)$ y $c=d(v,w)$, where $d$ denotes the Euclidean metric on $\mathbb{R}^{2}$.
\end{proposition}

\begin{theorem}\label{th6} Suppose $f:[0,\infty)\to [0,\infty)$ is amenable.
Then the following statements are equivalent.

{\rm (i)} $f\in  \mathcal{EB}$.

{\rm (ii)} There exists $\theta:X\times X\to [1,\infty)$ such that $(f(a),f(b),f(c))\in\triangle_{\theta}$ for all $(a,b,c)\in\triangle$.
\end{theorem}
\begin{proof} Assume that $f\in  \mathcal{EB}$.
Let $d$ be the Euclidean metric on $\mathbb{R}^{2}$.
Then $f\circ d$ is a extended $b$--metric.
So there exist $\theta:X\times X\to [1,\infty)$ such that
\[
(f\circ d)(x,y)\leq \theta(x,y)((f\circ d)(x,z)+(f\circ d)(z,y))\quad{\rm for}\quad{\rm all}\quad x,y,z\in \mathbb{R}^{2}.
\]
Let $(a,b,c)\in \triangle$. By the Proposition \ref{pro0}, there are $u,v,w\in \mathbb{R}^{2}$ such that $d(u,w)=a$, $d(u,v)=b$ and $d(v,w)=c$. Then
\[
f(a)=(f\circ d)(u,w)\leq \theta(u,w)((f\circ d)(u,v)+(f\circ d)(v,w))
\]
\[
\hspace{8mm}=\theta(u,w)(f(b)+f(c)).
\]
Similarly, $f(b)\leq \theta(u,v)(f(a)+f(c))$ and $f(c)\leq \theta(v,w)(f(a)+f(b)) $.
Therefore $(f(a),f(b),f(c))\in \triangle_{\theta}$.

For the converse, assume that there exists $\theta:X\times X\to [1,\infty)$ such that $(f(a),f(b),f(c))\in \triangle_{\theta}$ for all $(a,b,c)\in \triangle$.
Let  $(X,d_{\theta})$ be a extended $b$-metric space and let $x,y,z\in X$.
Since $f$ is amenable, $(f\circ d_{\theta})(x,y)=0$ if and only if $x=y$.
The condition $d_{\theta}2$ is obvious. So it remains to show that $(f\circ d_{\theta})_{\hat{ \theta}}$ satisfies $d_{\theta}3$.
Since
$(f(d(x,y)),f(d(x,z)),f(d(z,y)))\in\triangle_{\theta}$, for $(d(x,y), d(x,z),d(z,y))\in\triangle$, it follows that

\[
(f\circ d)(x,y)\leq \theta(x,y)((f\circ d)(x,z)+(f\circ d)(z,y)).
\]
Hence $f\circ d$ is a extended $b-$metric. This completes the proof.

\end{proof}

If we replace $\mathcal{EB}$ by $\mathcal{MB}$ in Theorem \ref{th6}, we obtain as a corollary the results given by Khemaratchatakumthorn and Pongsriiam in {\rm \cite[Theorem 17]{T.K}}.

\begin{corollary}\label{th5} Let $f:[0,\infty)\to [0,\infty)$.
If $f\in  \mathcal{MB}$, then $f$ is amenable and quasi-subadditive.
\end{corollary}
\begin{proof} Assume that $f\in  \mathcal{MB}$. By the theorems \ref{thkp} and \ref{th2}, we obtain that $f\in \mathcal{EB}$.
By the theorem \ref{thEB},  $f$ is amenable y quasi-subadditive.
\end{proof}

\begin{corollary}\label{th5} Suppose $f:[0,\infty)\to [0,\infty)$ is amenable.
Then the following statements are equivalent.

{\rm (i)} $f\in  \mathcal{MB}$.

{\rm (ii)} There exists $s\geq 1$ such that $(f(a),f(b),f(c))\in\triangle_{s}$ for all $(a,b,c)\in\triangle$.
\end{corollary}
\begin{proof} Assume that $f\in  \mathcal{MB}$. By the theorems \ref{thkp} and \ref{th2}, we obtain that $f\in \mathcal{EB}$.
Now by the Theorem \ref{th6}, there exists $\theta:X\times X\to [1,\infty)$ such that $(f(a),f(b),f(c))\in\triangle_{\theta}$ for all $(a,b,c)\in\triangle$. Suppose that, for each $s\geq 1$, there exists $(a,b,c)\in \triangle$ such that $(f(a),f(b),f(c))\not\in \triangle_{s}$.
Then
\[
f(a)>s(f(b)+f(c)),\quad f(b)>s(f(a)+f(c))\quad {\rm or}\quad f(c)>s(f(a)+f(b)).
\]
Let $s=\theta(s,t)$. We obtained
\[
f(a)>\theta(s,t)(f(b)+f(c)),\quad f(b)>\theta(s,t)(f(a)+f(c))
\]
or $f(c)>\theta(s,t)(f(a)+f(b))$.
Namely, $(f(a),f(b),f(c))\not\in\triangle_{\theta}$ for some $(a,b,c)\in\triangle$, which is a contradiction.

For the converse, see the proof proposed in {\rm \cite[Theorem 17]{T.K}}.
\end{proof}

The Theorem \ref{th2},  show the set of all
extended $b$-metric-preserving functions, contains the $b$-metric-preserving functions and these in turn to the ultrametric-preserving functions.

In the following theorem we will verify that, if $f$ belongs to the larger family and $\lim_{x\to 0^{+}}f(x)=a$
and $f(x)=a$ for all $x\in(0,b]$, then the graph $f$ is contained in  the region proposed by J, Dobos and Z, Piotrowski see \cite{D.P}.

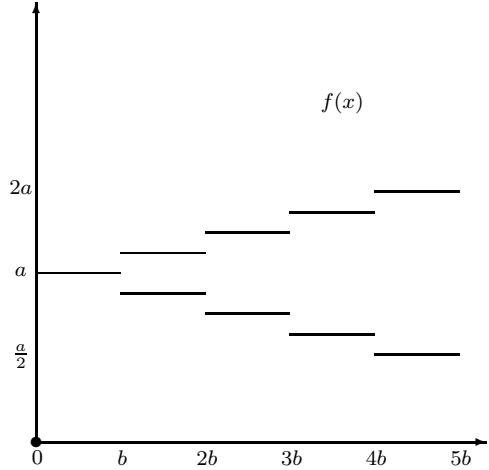
\begin{figure}[h]\label{fig6}
\unitlength=0.45mm \linethickness{0.7pt}
\begin{picture}(0,135)
\put(-40,0){\vector(1,0){133}}
\put(-40,0){\vector(0,1){130,0}}
\multiput(-40,50)(25,6){5}{\line(1,0){25}}
\multiput(-40,50)(25,-6){5}{\line(1,0){25}}
\put(50,100){\makebox(1,1)[1,1]{{\footnotesize $f(x)$}}}
\put(-45,75){\makebox(1,1)[1,1]{{\footnotesize $2a$}}}
\put(-45,50){\makebox(1,1)[1,1]{{\footnotesize $a$}}}
\put(-45,25){\makebox(1,1)[1,1]{{\footnotesize $\frac{a}{2}$}}}
\put(-15,-5){\makebox(1,1)[1,1]{{\footnotesize $b$}}}
\put(-40,-5){\makebox(1,1)[1,1]{{\footnotesize $0$}}}
\put(-40,0){\circle*{3.0}}
\put(10,-5){\makebox(1,1)[1,1]{{\footnotesize $2b$}}}
\put(35,-5){\makebox(1,1)[1,1]{{\footnotesize $3b$}}}
\put(60,-5){\makebox(1,1)[1,1]{{\footnotesize $4b$}}}
\put(85,-5){\makebox(1,1)[1,1]{{\footnotesize $5b$}}}
\end{picture}
\end{figure}

\begin{theorem}\label{thm20} Suppose that $f\in  \mathcal{BE}$,  $\lim_{x\to 0^{+}}f(x)=a$
and $f(x)=a$ for all $x\in(0,b]$, then for each $n\in \mathbb{N}$ and each $x\in (nb,(n+1)b]$,
$$\frac{a}{2}\leq f(x)\leq 2^na.$$
\end{theorem}
\begin{proof}
We apply the principle of mathematical induction.
Let us see that, for $ n = 1 $, the conclusion is satisfied
\[
\frac{a}{2}\leq f(x)\leq 2a\quad{\rm for}\quad{\rm all}\quad x \in (b,2b].
\]

Inequality first
\[
\frac{a}{2}\leq f(x)\quad{\rm para}\quad{\rm all}\quad  x \in (b,2b].
\]

Suppose that, there exist an  $x \in (b,2b]$ such that  $f(x)<\frac{a}{2}$.
Let $z \in (0,b]$.

Note that, $(x,x,z)$ is a triangle triplet, while  $(f(x),f(x),f(z))$ does not, since
\[
f(x)+f(x)<\frac{a}{2}+\frac{a}{2}=a=f(z).
\]
That is, for each,  $\theta:X\times X\to [1,\infty)$, there exist a triplet $(x,x,z)\in \triangle$ such that
\[
f(z)\not\leq \theta(x,z)(f(x)+f(x)).
\]
So by Theorem \ref{th6}, $f$ does not extended $b-$m\'etric preserving, which contradicts the hypothesis.

Now this other inequality $f(x)\leq 2a$ for all $x \in (b,2b]$.
Suppose that, there exist $x \in (b,2b]$ such that $f(x)>2a$.

Note that, $(\frac{x}{2},\frac{x}{2},x)$ is a triangle triplet, while
$$\Big(f\Big(\frac{x}{2}\Big),f\Big(\frac{x}{2}\Big),f(x)\Big)$$ does not,
since
\[
f\Big(\frac{x}{2}\Big)+f\Big(\frac{x}{2}\Big)=a+a=2a<f(x).
\]
That is, for each, $\theta:X\times X\to [1,\infty)$ there exist a triplet $(\frac{x}{2},\frac{x}{2},x)\in \triangle$ such that
$f(x)\not\leq \theta(\frac{x}{2},z)(f(\frac{x}{2})+f(\frac{x}{2}))$.
So by Theorem \ref{th6}, $f$ does not extended  $b-$m\'etric preserving, which contradicts the hypothesis.

(H.I) Assume that, for $n=k$ the inequality is satisfied,
\[
f(x)\leq 2^{k}a\quad{\rm for}\quad{\rm all}\quad  x\in (kb,(k+1)b].
\]
We will show that, for $n=k+1$ it is also satisfied.
Namely
\[
f(x)\leq 2^{k+1}a\quad{\rm for}\quad{\rm all}\quad x\in ((k+1)b,(k+2)b].
\]
Suppose that, there exist a element
$x\in ((k+1)b,(k+2)b]$ such that $f(x)>2^{k+1}a$.

Note that  $(\frac{x}{2},\frac{x}{2},x)$ is a triangle triplet and
$(f(\frac{x}{2}),f(\frac{x}{2}),f(x))$ does not, since by the inductive hypothesis
$$f\Big(\frac{x}{2}\Big)+f\Big(\frac{x}{2}\Big)\leq 2^{k}a+2^{k}a=2(2^{k}a)=2^{k+1}a<f(x).$$
That is, for each, $\theta:X\times X\to [1,\infty)$ there exist a triplet $(\frac{x}{2},\frac{x}{2},x)\in \triangle$ such that
$f(x)\not\leq \theta(\frac{x}{2},z)(f(\frac{x}{2})+f(\frac{x}{2}))$.
So by Theorem \ref{th6}, $f$ does not extended  $b-$metric preserving, which contradicts the hypothesis.

\end{proof}

\begin{remark} Observe that, if in the Theorem \ref{thm20}  we substitute the interval $(0,b]$ by $(0,1]$ we obtain: If $f\in \mathcal{BE}$, $\lim_{x\to 0^{+}}f(x)=a$
and $f(x)=a$ for each $x\in(0,1]$, then  for each $n\in \mathbb{N}$ and $x\in (n,(n+1)]$, $\frac{a}{2}\leq f(x)\leq 2^na$.
\end{remark}

\bibliographystyle{plain}

\end{document}